\documentclass{amsart}

\usepackage[utf8]{inputenc}
\usepackage{amsmath, amssymb}
\usepackage{graphicx}
\usepackage{subcaption}
\usepackage{soul}

\usepackage[dvips,%
    letterpaper,%
    includehead,%
    includefoot,%
    nomarginpar,%
    lmargin=1.5in,%
    rmargin=1.5in,%
    tmargin=1.5in,%
    bmargin=1.5in,%
  ]{geometry}

\usepackage{xcolor}
\usepackage[colorlinks,
    linkcolor={blue!60!black},
    citecolor={blue!60!black},
    urlcolor={blue!60!black}]{hyperref}

\usepackage{enumerate,enumitem}
\usepackage{tikz}
\usetikzlibrary{calc, matrix, arrows, cd, decorations.markings, knots}
\usetikzlibrary{decorations.pathmorphing}
\usetikzlibrary{decorations.pathreplacing}

\theoremstyle{plain}
\newtheorem{proposition}{Proposition}[section]
\newtheorem{theorem}{Theorem}[section]

\newtheorem{lemma}{Lemma}[section]

\theoremstyle{definition}
\newtheorem{definition}{Definition}[section]

\theoremstyle{remark}

\makeatletter
\let\c@corollary=\c@theorem
\let\c@proposition=\c@theorem
\let\c@lemma=\c@theorem
\let\c@remark=\c@theorem
\let\c@definition=\c@theorem
\let\c@construction=\c@theorem
\let\c@example=\c@theorem
\let\c@question=\c@theorem
\let\c@equation\c@theorem
\makeatother

\def\makeautorefname#1#2{\expandafter\def\csname#1autorefname\endcsname{#2}}
\makeautorefname{theorem}{Theorem}%
\makeautorefname{corollary}{Corollary}%
\makeautorefname{proposition}{Proposition}%
\makeautorefname{lemma}{Lemma}%
\makeautorefname{remark}{Remark}%
\makeautorefname{definition}{Definition}%
\makeautorefname{section}{Section}%
\makeautorefname{construction}{Construction}%
\makeautorefname{example}{Example}%
\makeautorefname{question}{Question}%

\newcommand{\sm}{\setminus}

\newcommand{\N}{\mathbb{N}}

\newcommand{\R}{\mathbb{R}}

\newcommand{\DD}{\mathcal{D}}

\newcommand{\Int}{\operatorname{Int}}

\newcommand{\diam}{\operatorname{diam}}

\renewcommand{\epsilon}{\varepsilon}
\renewcommand{\phi}{\varphi}
\DeclareMathOperator{\id}{Id}	

\begin{document}
\title{Null, recursively starlike-equivalent decompositions shrink}

\author[J.\ Meier]{Jeffrey Meier}
\address{Department of Mathematics, Western Washington University, USA}
\email{jeffrey.meier@wwu.edu}

\author[P. Orson]{Patrick Orson}
\address{Department of Mathematics, Boston College, USA}
\email{patrick.orson@bc.edu}

\author[A. Ray]{Arunima Ray}
\address{Max-Planck-Institut f\"{u}r Mathematik, Bonn, Germany}
\email{aruray@mpim-bonn.mpg.de}

\begin{abstract}
A subset $E$ of a metric space $X$ is said to be \emph{starlike-equivalent} if it has a neighbourhood which is mapped homeomorphically into $\mathbb{R}^n$ for some $n$, sending $E$ to a starlike set. A subset $E\subset X$ is said to be \emph{recursively starlike-equivalent} if it can be expressed as a finite nested union of closed subsets $\{E_i\}_{i=0}^{N+1}$ such that $E_{i}/E_{i+1}\subset X/E_{i+1}$ is starlike-equivalent for each $i$ and $E_{N+1}$ is a point. A decomposition $\mathcal{D}$ of a metric space $X$ is said to be recursively starlike-equivalent, if there exists $N\geq 0$ such that each element of $\mathcal{D}$ is recursively starlike-equivalent of filtration length $N$. 

We prove that any null, recursively starlike-equivalent decomposition $\mathcal{D}$ of a compact metric space $X$ shrinks, that is, the quotient map $X\to X/\mathcal{D}$ is the limit of a sequence of homeomorphisms.  This is a strong generalisation of results of Denman-Starbird and Freedman and is applicable to the proof of Freedman's celebrated disc embedding theorem. The latter leads to a multitude of foundational results for topological $4$-manifolds, including the $4$-dimensional Poincar\'{e} conjecture.
\end{abstract}

\maketitle

\section{Introduction}

A collection $\mathcal{D}$ of pairwise disjoint subsets of a metric space $X$ is called a \emph{decomposition} of $X$. The field of \emph{decomposition space theory} is concerned with the question of whether the space $X$ is homeomorphic to the quotient $X/\mathcal{D}$, obtained by collapsing the elements of $\mathcal{D}$ to distinct individual points. More generally, one can ask whether the quotient map $X\to X/\mathcal{D}$ is \emph{approximable by homeomorphisms}, that is, whether there exists a sequence of homeomorphisms $X\to X/\mathcal{D}$ converging, with respect to the uniform metric on the set of continuous functions from $X$ to $X/\mathcal{D}$, to the quotient map. In the latter case, we say that the decomposition~$\mathcal{D}$ \emph{shrinks}.  

Decomposition space theory has a long, storied history, for which we refer the reader to the book of Daverman~\cite{Daverman}. Landmark applications of the theory include the double suspension theorem of Cannon and Edwards~\cite{Cannon, Edwards:doublesuspension} and the disc embedding theorem of Freedman~\cite{F, Freedman:1984-1}, later generalised by Freedman and Quinn~\cite{FQ}.

The original disc embedding theorem of Freedman~\cite{F} is for simply connected $4$-manifolds. The book of Freedman and Quinn~\cite{FQ} extends the theorem to certain non-simply connected $4$-manifolds, namely those with \emph{good} fundamental group~\cite[Definition~11.23]{Freedman-notes}.
A key step in the proof of the general version shows that the quotient of the $4$-dimensional disc $D^4$ by a decomposition consisting of `red blood cells' (see Figure~\ref{fig:red-blood-cell}) is homeomorphic to $D^4$ (see Theorem~\ref{thm:alpha_ABH} below for a precise statement). The original proof by Freedman has a similar component; however, instead of red blood cells, Freedman considered so-called `hairy red blood cells', which are simply red blood cells with a Cantor set's worth of generalised Whitehead continua attached. Freedman showed that his decomposition shrinks. The book of Freedman and Quinn gives a different proof, using relations, and simply shows that the quotient of $D^4$ by the decomposition consisting of red blood cells is homeomorphic to $D^4$. It is not yet known whether the disc embedding theorem can be generalised to apply to all $4$-manifolds.

\begin{figure}[htb]
\centering
\includegraphics[width=5cm]{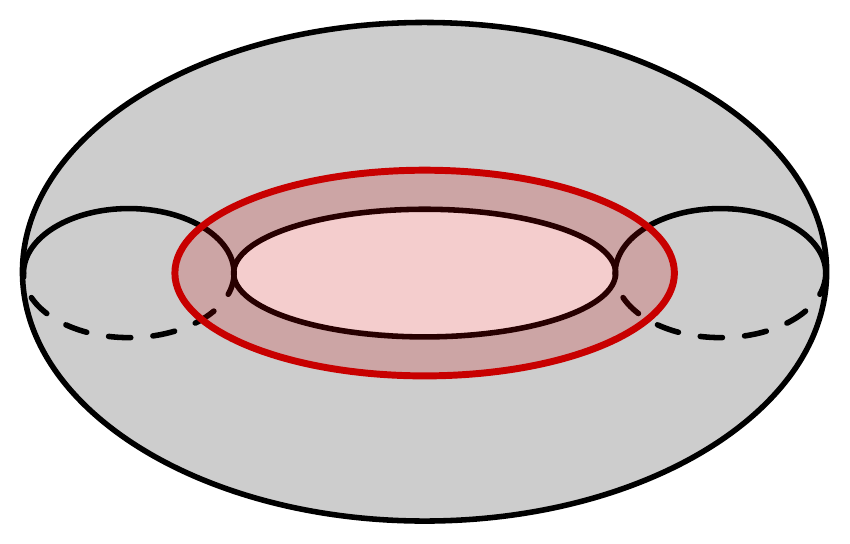}
\caption{A model red blood cell consists of a copy of $S^1\times D^3$ along with a disc glued in along $S^1\times \{*\}$, where $*$ is a point in $\partial D^3$. The fourth dimension is suppressed in the figure. A red blood cell in $D^4$ is the image of a flat, injective, continuous map of a model red blood cell. }
\label{fig:red-blood-cell}

\end{figure}

A subset $E$ of a metric space $X$ is said to be \emph{starlike-equivalent} if it has a neighbourhood which is mapped homeomorphically into $\mathbb{R}^n$ for some $n$ sending $E$ to a starlike set. A subset~$E\subset X$ is said to be \emph{recursively starlike-equivalent} if it can be expressed as a finite nested union of closed subsets   
\[\{e\}=E_{N+1}\subset E_N\subset E_{N-1}\subset \cdots\subset E_{1}\subset E_0=E,\]
	such that $E_{i}/E_{i+1}\subset X/E_{i+1}$ is starlike-equivalent for each $i$. The parameter $N$ is called the \emph{filtration length} of the recursively starlike-equivalent set $E$.
	It is then straightforward to see that the red blood cells of Freedman and Quinn are recursively starlike-equivalent of filtration length one. Recursively starlike-equivalent sets were called \emph{eventually starlike-equivalent sets} in~\cite[p.~78]{FQ}. A decomposition $\mathcal{D}$ of a metric space $X$ is said to be recursively starlike-equivalent of length~$N$ if each element of $\mathcal{D}$ is recursively starlike-equivalent of filtration length $N$. Lastly, a decomposition $\mathcal{D}$ of a metric space $X$ is said to be \emph{null} if for any~$\varepsilon>0$ only finitely many elements of $\mathcal{D}$ have diameter greater than $\varepsilon$. 

The following is the main theorem of this paper.  The application of this theorem to the proof of the disc embedding theorem is given in Theorem~\ref{thm:alpha_ABH} below.
	
\begin{theorem}\label{thm:main-theorem}
Let $\mathcal{D}$ be a null, recursively starlike-equivalent decomposition of length $N\geq 0$ of a compact metric space $X$, and let $U\subset X$ be an open set such that all the non-singleton elements of $\mathcal{D}$ lie in $U$. Then the quotient map $\pi\colon X\to X/\DD$ is approximable by homeomorphisms agreeing with $\pi$ on $X\sm U$. 
\end{theorem}	
	
Freedman's shrink in~\cite{F} did not apply a shrinking result as general as the above. Rather, he argued as far as the statement that null, starlike-equivalent decompositions shrink and then gave an \textit{ad hoc} argument for a $3$-stage shrink; see~\cite[Section~8]{F}. A similarly \textit{ad hoc} argument was given in~\cite{FQ}, using relations, to show that the quotient of $D^4$ by the relevant decomposition consisting of red blood cells is homeomorphic to $D^4$. A weaker version of the statement of the above result, namely that the quotient space is homeomorphic to the original space, was left as an exercise in~\cite[Section~4.5]{FQ}. 

The general shrinking result above adds to a collection of standard ideas in the well-developed literature of decomposition space theory, beginning with work of Bean, who proved that null, starlike-equivalent decompositions of $\R^3$ shrink~\cite{Bean_starlike}. Daverman outlined an extension of this theorem to $\R^n$ for all $n$~\cite[Theorem~II.8.6]{Daverman}.
Any null decomposition is countable, and there has also been substantial interest in showing whether countable decompositions shrink. Bing showed that countable, starlike, upper semi-continuous decompositions of $\R^n$ shrink for all $n$~\cite{bingusc} (see also~\cite[Theorem~8.7]{Daverman}). Denman and Starbird showed in~\cite{SD} that countable, starlike-equivalent, upper semi-continuous decompositions of~$\mathbb{R}^3$ shrink. In the same paper, they introduced the notion of a \emph{birdlike-equivalent set}, a special type of recursively starlike-equivalent set, and showed that countable, upper semi-continuous decompositions of~$\mathbb{R}^3$ consisting of birdlike-equivalent sets shrink. It is to this body of shrinking results that Theorem~\ref{thm:main-theorem} rightly belongs.

\subsection*{Notation}
If $S$ is a subset of a topological space $X$, a \emph{neighbourhood} of $S$ in $X$ is a subset $V\subset X$ such that there is an open subset $U$ with $S\subset U\subset V$. For a metric space $(X,d)$, a point $x_0\in X$ and $r>0$, we write
\[
B(x_0;r):=\{x\in X\mid d(x,x_0)<r\}.
\]
We use $D^n\subset \R^n$ to denote the closed unit disc and $S^{n-1}$ to denote its boundary. 

\subsection*{Acknowledgements}
This paper arose from a chapter of~\cite{Freedman-notes} written by the same authors. The latter includes work by a number of contributors, and is based on lectures by Michael Freedman at the University of California Santa Barbara in 2013. We worked on this paper while all three of us were guests at the Hausdorff Research Institute for Mathematics, participating in the Junior Trimester Program in Topology during Autumn 2016, as well as while PO and JM visited AR at the Max Planck Institute for Mathematics. We thank HIM and MPIM for their hospitality. We also thank Mark Powell for numerous delightful conversations. JM is partially supported by NSF grant DMS-1933019.

\section{Preliminaries}

\noindent In this section we recall some foundational definitions and results from decomposition space theory.

\subsection{Upper semi-continuous decompositions}\label{sec:decompositions}

A \emph{decomposition} $\DD$ of a topological space $X$ is a collection of pairwise disjoint subsets of $X$.  In what follows $\DD = \{\Delta_i\}_{i\in I}$ will always denote a decomposition of a topological space $X$, and $\pi\colon X\to X/\DD$ will always denote the quotient map obtained by collapsing each $\Delta_i$ to a distinct point.

\begin{definition}\label{def:saturated}
Given any subset~$S\subset X$ we define its \emph{{$\DD$-}saturation}
 as
\[
\pi^{-1} (\pi(S))=S \cup \bigcup \{\Delta_i\mid \Delta_i\cap S \neq \emptyset\}.
\]
We say that~$S$ is \emph{$\DD$-saturated} if~$S=\pi^{-1}(\pi(S))$.

For any subset $S \subset X$,
we also consider the largest $\DD$-saturated subset of $S$, defined by
\[S^*:= S \sm \pi^{-1}\big (\pi(X \sm S) \big ) = S\sm \bigcup \{\Delta_i\mid \Delta_i \not\subset S\}.\]
\end{definition}

\begin{definition}\label{def:upper semi-continuity}
A decomposition~$\DD$ of a topological space $X$ is said to be \emph{upper semi-continuous}
if for each open subset $U \subset X$, the set
$U^*$ is also open and each decomposition element $\Delta$ is closed and compact.
\end{definition}

We prefer to work with upper semi-continuous decompositions due to the following proposition. 

\begin{proposition}[{\cite[Proposition~2.2]{Daverman}}]\label{prop:pointsetfun}
If $\DD$ is an upper semi-continuous decomposition of a compact metric space~$X$, then $X/\DD$ is a compact and metrisable.
\end{proposition}

\subsection{Shrinking}

Let $X$ and $Y$ be compact metric spaces. We denote the metric space of continuous functions from $X$ to $Y$, equipped with the uniform metric, by $\mathcal{C}(X,Y)$. Recall that the uniform metric is defined by setting $d(f,g)=\sup_{x \in X} d_Y(f(x),g(x))$ for two functions $f,g \colon X \to Y$.  It is also known that $\mathcal{C}(X,Y)$ is a complete metric space~\cite[Theorems~43.6~and~45.1]{Munkres-Topology}. Observe that the metric space $\mathcal{C}(X,X)$ contains the subspace~$\mathcal{C}_A(X,X)$
of functions $f \colon X \to X$ with $f|_{A} = \id_{A}$, for any subset $A\subset X$. This is a closed set in $\mathcal{C}(X,X)$ and thus is itself a complete metric space under the induced metric. The following definition formalises the notion of approximating a function by homeomorphisms.

\begin{definition}\label{def:ABH}
Let $X$ and $Y$ be compact metric spaces and let
$f\colon X\to Y$ be a surjective continuous map.
The map $f$ is said to be \emph{approximable by homeomorphisms}\index{approximable by homeomorphisms} if
there is a sequence of homeomorphisms $\{h_n \colon X \to Y\}_{n=1}^{\infty}$ that
converges to $f$ in $\mathcal{C}(X,Y)$.
\end{definition}

A celebrated theorem of Bing~\cite{Bing1952} (see also~\cite{Edwards_ICM}) characterises when functions between compact metric spaces are approximable by homeomorphisms; this is called the \emph{Bing shrinking criterion}. Recall that a decomposition is said to shrink if the corresponding quotient map is approximable by homeomorphisms. The following definition formalises this notion. Specifically, conditions (i) and (ii) refer to the Bing shrinking criterion. 

\begin{definition}\label{def:shrinkability, ABH}
Let~$\DD$ be an upper semi-continuous decomposition of a compact metric space~$X$.
The decomposition $\DD$ is \emph{shrinkable}
if for every $\varepsilon>0$ there exists a self-homeomorphism $H\colon X\to X$ such that
\begin{enumerate}
	\item[(i)] for all $x \in X$, we have that $d_{X / \DD}\big (\pi(x), \pi\circ H(x) \big ) < \varepsilon$, and
	\item[(ii)] for all $\Delta\in \DD$, we have~$\diam_X(H(\Delta))<\varepsilon$,
\end{enumerate}
where $d_{X / \DD}$ is some chosen metric on $X / \DD$.

Given an open set $W\subset X$ containing every $\Delta\in \mathcal{D}$ with $\lvert \Delta\rvert >1$, we say $\DD$ is \emph{shrinkable fixing $X \sm W$} if for every $\varepsilon>0$
there exists a homeomorphism $H \colon X \to X$ satisfying conditions~(i) and~(ii) which fixes the set $X \sm W$ pointwise.

We say the decomposition $\DD$ is \emph{strongly shrinkable} if $\DD$ is shrinkable fixing $X \sm W$ for every open set $W\subset X$ of the form above.
\end{definition}

By the Bing shrinking criterion, an upper semi-continous decomposition $\DD$ of a compact metric space~$X$ is shrinkable if and only if the quotient map $\pi\colon X\to X/\DD$ is approximable by homeomorphisms.

When $X$ is a manifold, shrinkability and strong shrinkability are equivalent notions~\cite[p.~108,~Theorem~13.1]{Daverman}.  A counterexample to the general statement for non-manifolds can be constructed using~\cite[Examples~7.1 and~7.2]{Daverman}, as indicated in~\cite[p.~112,~Exercise~13.2]{Daverman}.

We will need the following fact. 
\begin{proposition}\label{prop:composition-of-ABH}
  Let $f \colon X \to Y$ and $g \colon Y \to Z$ be maps between compact metric spaces that are approximable by homeomorphisms.  Then $g \circ f \colon X \to Z$ is also approximable by homeomorphisms.
\end{proposition}

\begin{proof}   The proof is straightforward analysis using the Heine-Cantor theorem.  

\end{proof}

\subsection{Null collections and starlike sets}\label{sec:starlike}

\begin{definition}\label{def:null}
A collection of subsets~$\{T_i\}_{i\in I}$ of a metric space~$X$ is said to be \emph{null} if for any~$\varepsilon>0$ there are only finitely many~$i\in I$ such that~$\diam(T_i)>\varepsilon$.
\end{definition}

An immediate consequence of the definition is that the number of non-singleton elements in any null collection must be countable. This is easily seen by considering a (countable) sequence $\{\varepsilon_i\}$ converging to zero. Any null decomposition $\DD$ for a metric space $(X,d)$ is upper semi-continuous if each $\Delta\in \DD$ is compact~\cite[Proposition~2.3]{Daverman}.

\begin{definition}\label{def:starlike}\label{def:equivalent}

	A subset $E$ of $\Int D^n$ is said to be \emph{starlike} if there is a point $O_E\in E$ and an upper semi-continuous function $\rho_E\colon S^{n-1}\to [0,\infty)$ such that
	$$E=\{O_E+t\xi\mid \xi\in S^{n-1},\,0\leq t\leq \rho_E(\xi)\}.$$ 
	Then $O_E$ and $\rho_E$ are called the \emph{origin} and the \emph{radius function} for $E$ respectively.

A subset $E$ of a topological space $X$ is said to be \emph{starlike-equivalent} if, for some $n$, there is a map $f\colon N(E)\to \Int{D^n}$ from a closed neighbourhood of $E$, which is a homeomorphism to its image and such that $f(E)$ is starlike. An \emph{origin} for such an $E$ is $f^{-1}(O_E)$, where $O_E$ is an origin for $f(E)$. Note that $f$, and thus the origin, need not be unique. 

A subset $E$ of a topological space $X$ is said to be \emph{recursively starlike-equivalent} (with \emph{filtration length} $N$) if there exists a finite filtration
	\[ \{e\} =E_{N+1}\subset E_N\subset E_{N-1}\subset \cdots\subset E_{1}\subset E_0=E\]
	by closed subsets $E_i$, such that $E_{i}/E_{i+1}\subset X/E_{i+1}$ is starlike-equivalent for each $0\leq i \leq N$ and such that $E_{i+1}/E_{i+1}$ may be taken to be the origin of $E_i/E_{i+1}$. 
\end{definition}

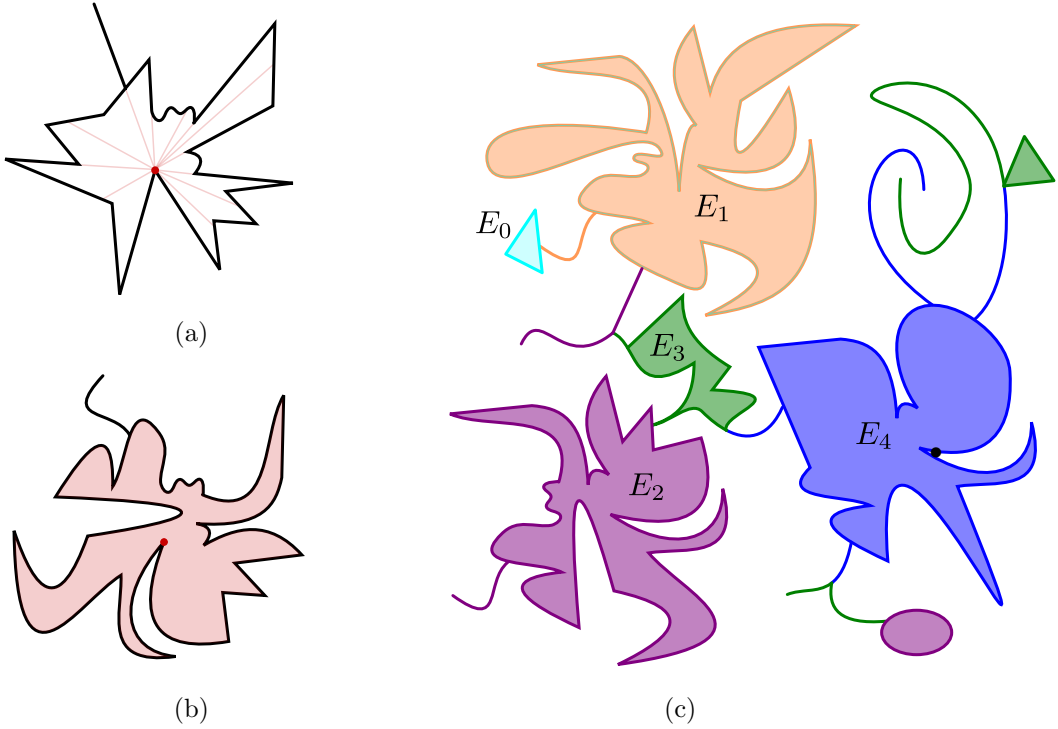
\begin{figure}[htb]
\begin{tikzpicture}

\node[inner sep=0pt] (types) at (0,0){\includegraphics[width=\textwidth]{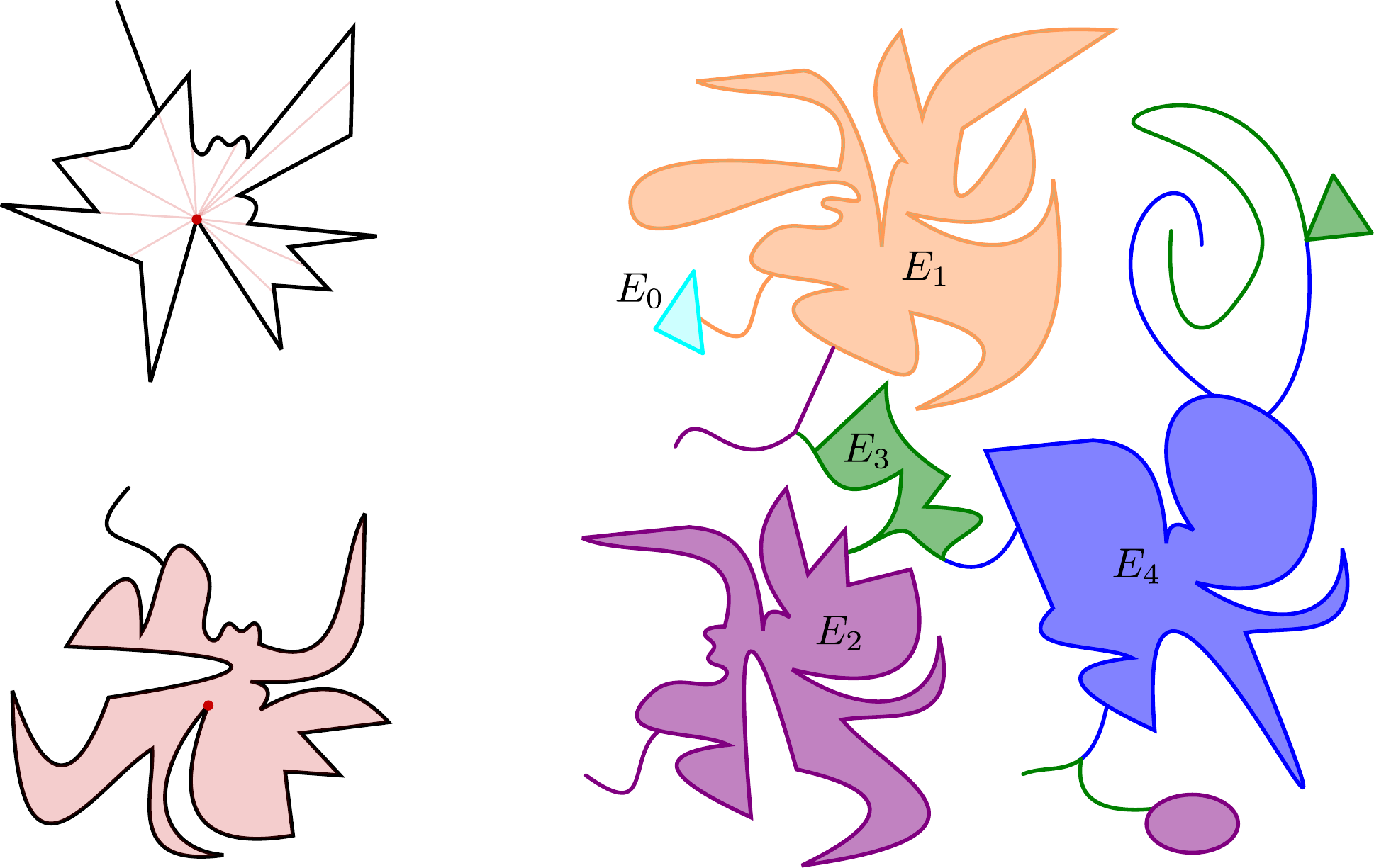}};
\node at (-4.5,0) {(a)};
\node at (-4.5,-5) {(b)};
\node at (2,-5) {(c)};
\node at (5.4,-1.58) {$\bullet$};
\end{tikzpicture}
	\caption[Examples of starlike, starlike-equivalent, and recursively starlike-equivalent sets.]{Examples of (a) starlike, (b) starlike-equivalent and (c) recursively starlike-equivalent sets. In (b), the origin is shown in red. In (c), $E_4$ is shown in dark blue, $E_3$ is the union of $E_4$ and the space shown in green, $E_2$ is the union of $E_3$ and the space shown in purple, and $E_1$ is the union of $E_2$ and the space shown in orange. The entire space is $E$; the black dot in $E_4$ is $e$.}
  \label{fig:starliketypes}
\end{figure}

\textit{A priori} if a set $E\subset X$ is recursively starlike-equivalent as in Definition~\ref{def:equivalent} then for $i=1,2,\dots,N$ there exist $n_i\in\N$ and maps $f_i\colon N(E_i/E_{i+1})\to D^{n_i}$ such that $f_i$ is a homeomorphism to its image and that the image of $E_i/E_{i+1}$ is starlike. In fact, these $n_i$ must all be equal as we now show. 

Fix $i\in\{1,2,\dots,N\}$. The set $E_i/E_{i+1}\subset X/E_{i+1}$ is shrinkable, by~\cite[Theorem 8.6]{Daverman}, since it is starlike-equivalent by definition. Thus, the quotient map 
\[\pi\colon X/E_{i+1}\to (X/E_{i+1})/(E_i/E_{i+1})=X/E_i\]
 is approximable by homeomorphisms. Note that $\pi(E_{i+1}/E_{i+1})=E_i/E_i$. Choose $\varepsilon>0$ small enough so that the ball $B(E_i/E_i;\varepsilon)$ lies within $N(E_{i-1}/E_i)$. Choose a homeomorphism $g$, approximating $\pi$, so that $g(E_{i+1}/E_{i+1})\in B(E_i/E_i;\varepsilon)$. Let $V$ be a neighbourhood of $E_i/E_{i+1}$ in $X/E_{i+1}$. Then $g(V)$ intersects the set $N(E_{i-1}/E_i)$ since $g(V)$ contains $g(E_{i+1}/E_{i+1})\in B(E_i/E_i;\varepsilon)\subset N(E_{i-1}/E_i)$. 

Thus, there is some nonempty open set $U$ contained within the intersection of the sets $g(N(E_i/E_{i+1}))$ and $N(E_{i-1}/E_i)$. Via $f_{i-1}$, this $U$ is homeomorphic to a non-empty open subset of $D^{n_{i-1}}$ and via $f_i\circ g^{-1}$ it is homeomorphic to a non-empty open subset of $D^{n_{i}}$. By invariance of domain, we see that $n_i=n_{i-1}$ for all $i\in\{1,2,\dots,N\}$ as claimed.

In light of this observation, we will henceforth consider the maps $f_i$ arising in Definition~\ref{def:starlike} as having a fixed codomain $D^n$.
We end this section with the following lemma. 

\begin{lemma}\label{lem:RSE-preserved}
	Let $E$ be a recursively starlike-equivalent set with filtration length $N$ in the compact metric space $X$. Let $Y$ be a topological space. Let $h\colon X\to Y$ be a quotient map which is injective on some open neighbourhood $U\supset E$. Then $h(E)\subset Y$ is recursively starlike-equivalent with filtration length $N$. 
\end{lemma}

\begin{proof}
By hypothesis, there exists a filtration 
\[\{e\}=E_{N+1}\subset E_N\subset E_{N-1}\subset \cdots\subset E_{1}\subset E_0=E\]
where each $E_i$ is closed, and there exist closed neighbourhoods $N(E_i/E_{i+1})\subset X/E_{i+1}$ and functions $f_{i+1}\colon N(E_i/E_{i+1})\to D^n$ for some fixed $n$ which are homeomorphisms onto their image and send $E_i/E_{i+1}$ to a starlike set. By choosing smaller sets if necessary, assume that $N(E_i/E_{i+1})\subset U/E_{i+1}$. This uses the fact that every metric space is normal.

Next, observe that the function $h$ descends to give the quotient maps 
\[h_{i+1}\colon X/E_{i+1}\to Y/h(E_{i+1}).\]
 Then the restrictions $h_{i+1}\vert_{U/E_{i+1}}\colon U/E_{i+1}\to h(U)/h(E_{i+1})$ are bijective quotient maps and thus homeomorphisms. Now we simply define the filtration
\[\{h(e)\}=h(E)_{N+1}\subset h(E)_N\subset h(E)_{N-1}\subset \cdots\subset h(E)_{1}\subset h(E)_0=h(E)\]
where we set $h(E)_i=h(E_i)$ for all $i$. We define the neighbourhood $N(h(E)_i/h(E)_{i+1})$ as the set $h_{i+1}(N(E_i/E_{i+1}))$ for all $i$. Moreover, we have the function 
\[
f_{i+1}\circ h_{i+1}\vert_{U/E_{i+1}}^{-1} \colon N(h(E)_i/h(E)_{i+1})\to D^n
\]
 which is a homeomorphism onto its image and maps $h(E)_i/h(E)_{i+1}$ to a starlike set by construction, for each $i$. This completes the proof.
\end{proof}

\section[Shrinking null, starlike decompositions]{Shrinking null, recursively starlike-equivalent decompositions}\label{sec:starlikeshrink}

We have the following key lemma for shrinking a starlike set, while controlling the effect of the shrink on a null collection. 

\begin{lemma}[Starlike shrinking lemma {\cite[Theorem~8.5]{Daverman}}]\label{lem:FQshrinking}
	 Let $(E,U,\mathcal{T},\varepsilon)$ be a $4$-tuple consisting of a starlike set $E\subset \Int D^n$, a null collection of closed sets $\mathcal{T}=\{T_j\}_{j\geq 1}$ in $\Int D^n\sm E$, an open set $U\subset \Int D^n$ containing $E$, and a real number $\varepsilon>0$.  Then there is a self--homeomorphism $h$ of $D^n$ such that\begin{enumerate}
	\item \label{item:FQshrinking1} $h$ is the identity on $D^n\setminus U$,
	\item \label{item:FQshrinking2}$\diam{h(E)}<\varepsilon$, and
	\item \label{item:FQshrinking3} for each $j\geq 1$, either $\diam{h(T_j)}<\varepsilon$ or $h$ fixes $T_j$ pointwise.
\end{enumerate}
\end{lemma}

Unsurprisingly, we have an analogue of the above result for starlike-equivalent sets, as we show next. 

\begin{lemma}\label{lem:FQshrinking2}
	Let $X$ be a compact metric space and $(E,U,\mathcal{T},\varepsilon)$ be a $4$-tuple consisting of a starlike-equivalent set $E\subset X$, an open set $U\subset X$ containing $E$, a null collection of closed sets $\mathcal{T}=\{T_j\}_{j\geq 1}$ in $X\sm E$, and a real number $\varepsilon>0$. Then there is a self-homeomorphism $h$ of $X$ such that
	\begin{enumerate}
	\item \label{item:FQshrinking21} $h$ is the identity on $X\setminus U$,
	\item \label{item:FQshrinking22}$\diam{ h(E)}<\varepsilon$, and
	\item \label{item:FQshrinking23} for each $j\geq 1$, either $\diam{h(T_j)}<\varepsilon$ or $h$ fixes $T_j$ pointwise.
\end{enumerate}
\end{lemma}

\begin{proof}
By definition, there exists a function $f\colon N(E)\to D^n$ from a closed neighbourhood of $E$ in $X$ which is a homeomorphism to its image and such that $f(E)$ is starlike. Assume, by choosing a slightly smaller $N(E)$ if necessary, that $N(E)\subset U$. This uses the fact that a metric space is normal.

Since $X$ is compact and each element of $\mathcal{T}$ is closed, we see that each element of $\mathcal{T}$ is compact. It is also clear that any starlike-equivalent set is compact.
Then we see that the decomposition $\mathcal{T}\cup \{E\}$ is null and thus upper semi-continuous by~\cite[Proposition 2.3]{Daverman}. Consequently, there exists some open set $V$ such that $E\subset V\subset N(E)$ and $V$ is $\mathcal{T}$-saturated; in particular, we could choose $V=N(E)^*$ and note that $E\subset N(E)^*$ since $E\subset N(E)$. In other words, for any $T\in \mathcal{T}$, $T\cap V\neq \emptyset$ implies that $T\subset V$. 

Since $N(E)$ is closed in the compact space $X$, it is compact and thus so is $f(N(E))$. 
By the Heine-Cantor theorem, the function $f$ as well as the function $f^{-1}$ on $f(N(E))$ is uniformly continuous. 
In particular, there exists $\delta>0$ such that for $x,y\in f(N(E))$, if $d(x,y)<\delta$ then $d(f^{-1}(x),f^{-1}(y))<\varepsilon$.

Let $J$ be the set such that $j\in J$ implies that $T_j\cap V\neq \emptyset$ (or equivalently, $T_j\subset V$). By the uniform continuity of $f$, the collection $\{f(T_j)\mid j\in J\}$ is a null collection in $\Int D^n\sm f(E)$. 

Write $h'\colon D^n\to D^n$ for the homeomorphism obtained by applying Lemma \ref{lem:FQshrinking} to the 4--tuple ($f(E),f(V),\{f(T_j)\mid j\in J\},\delta)$, where $\delta$ is the value obtained from the Heine-Cantor theorem above. As $h'$ is the identity on $f(N(E))\sm f(V)$, the self-homeomorphism $h:=f^{-1}\circ h'\circ f$ of $N(E)$ extends to a self-homeomorphism $h$ of $X$ by declaring $h$ to be the identity on $X\sm N(E)$. By construction, this $h$ has the required properties (1), (2), and (3).
\end{proof}

We now prove that null decompositions of a compact metric space consisting of starlike-equivalent sets are shrinkable. 

\begin{theorem}\label{thm:SE-shrinking}
Suppose that $\mathcal{D}$ is a null decomposition of a compact metric space $X$ consisting of starlike-equivalent sets and that there exists an open set $U\subset X$ such that all the non-singleton elements of $\mathcal{D}$ lie in $U$. 
Then the quotient map $\pi\colon X\to X/\DD$ is approximable by homeomorphisms agreeing with $\pi$ on $X\sm U$. 
\end{theorem}
\begin{proof}
Fix $\varepsilon >0$. Then since $\mathcal{D}$ is null, there exist finitely many $\Delta^1, \Delta^2, \cdots, \Delta^k\in\mathcal{D}$ with $\diam{\Delta^i}\geq \varepsilon$.  By Proposition~\ref{prop:composition-of-ABH}, we know that $X/\mathcal{D}$ is a compact metric space since $\mathcal{D}$ is upper semi-continuous. Fix a metric on $X/\mathcal{D}$.

Let $\underline{W}^1$ be an open neighbourhood of $\pi(\Delta^1)\in X/\mathcal{D}$ so that $\diam{\underline{W}^1}< \varepsilon$, $\underline{W}^1\subset \pi(U)$, and $\underline{W}^1\cap \pi(\Delta^j)=\emptyset$ for all $1<j\leq k$. Let $W^1=\pi^{-1}(\underline{W}^1)$. Note that $W^1$ is an open neighbourhood of $\Delta^1$ with $\Delta^1\subset W^1\subset U$ and $W^1\cap \Delta^j=\emptyset$ for all $1<j\leq k$.  Apply Lemma~\ref{lem:FQshrinking2} to the $4$-tuple $(\Delta^1, W^1, \mathcal{D}\setminus \{\Delta^1\}, \varepsilon)$ to obtain the homeomorphism $h^1\colon X\to X$. Note that $h^1$ fixes $\Delta^j$ pointwise for all $1<j\leq k$. 

Now we iterate. For each $i$, choose an open neighbourhood $\underline{W}^i$ of $\pi\circ h^{i-1}\circ \cdots \circ h^1(\Delta^i)$ in $X/\mathcal{D}$ so that 
\begin{enumerate}
\item $\underline{W}^i\cap \underline{W}^j=\emptyset$ for all $j\leq i-1$,
\item $\diam{\underline{W}^i}<\varepsilon$, and
\item $\underline{W}^i\subset \pi\circ h^{i-1}\circ \cdots \circ h^1(U)=\pi(U)$.
\end{enumerate}
Let $W^i=\pi^{-1}(\underline{W}^i)$. Apply Lemma~\ref{lem:FQshrinking2} to the $4$-tuple $(\Delta^i, W^i, \mathcal{D}^i\setminus \{\Delta^i\}, \varepsilon)$, where $\mathcal{D}^i$ is defined to be $\{h^{i-1}\circ \cdots \circ h^1(\Delta)\mid \Delta\in\mathcal{D}\}$ to produce a homeomorphism $h^i\colon X\to X$. The decomposition $\mathcal{D}^i$ is null since each $h^j$ for $j\leq i-1$ is uniformly continuous by the Heine-Cantor theorem. 

Now consider the composition $H=h^k\circ\cdots\circ h^1$. By construction, $H\vert_{X\setminus U}$ is the identity because $h^i$ is the identity on $X\setminus W^i\supset X\setminus U$ for each $i$. We claim that $\mathcal{D}$ satisfies Definition~\ref{def:shrinkability, ABH} using $H$. By construction, $\diam{H(\Delta)}<\varepsilon$ for all $\Delta\in \mathcal{D}$. Moreover, for all $x\in X$, either $H(x)=x$ or $x\in W^i$ for some $i$. If $x\in W^i$ for some $i$, then $\pi(x),\pi\circ H(x)\in \pi(W^i)=\underline{W}^i$. Then, since $\diam{\underline{W}^i}<\varepsilon$ by construction, we see that $d(\pi(x),\pi\circ H(x))<\varepsilon$. Thus, for all $x\in X$, we have that $d(\pi(x),\pi\circ H(x))<\varepsilon$, or in other words, $d(\pi,\pi\circ H)<\varepsilon$. As a result, we see that $\mathcal{D}$ shrinks and the map $\pi\colon X\to X/\mathcal{D}$ is approximable by homeomorphisms restricting to the identity on $X\setminus U$.  
\end{proof}

\noindent We are finally ready to prove the main theorem.

\begin{proof}[Proof of Theorem~\ref{thm:main-theorem}]
The proof proceeds by induction on $N$. When $N=0$, the decomposition elements are all starlike-equivalent, in which case the result follows from Theorem~\ref{thm:SE-shrinking}. 

For the inductive step, suppose that the conclusion of the theorem holds when the filtration length is $N-1$ for some $N\geq 1$. Suppose that $\mathcal{D}$ is a decomposition of $X$ consisting of recursively starlike-equivalent sets of filtration length $N\geq 1$. That is, for each $\Delta\in\mathcal{D}$, we have a filtration 
\[
\{e\in \Delta\}=\Delta_{N+1}\subset \Delta_N\subset \Delta_{N-1}\subset \cdots\subset \Delta_{1}\subset \Delta_0=\Delta.
\]
Consider the decomposition $\mathcal{D}_1=\{\Delta_1\mid \Delta\in \mathcal{D}\}$. By the inductive hypothesis, the quotient map $\pi_1\colon X\to X/\mathcal{D}_1$ is approximable by homeomorphisms restricting to the identity on $X\setminus U$. Here we are using the fact that $\mathcal{D}_1$ is a null decomposition, since $\mathcal{D}$ is. 

Next we consider the decomposition $\mathcal{D}/\mathcal{D}_1:=\{\pi_1(\Delta)\mid \Delta\in \mathcal{D}\}$ of the space $X/\mathcal{D}_1$. Note that the latter is a compact metric space since $\mathcal{D}_1$ is upper semi-continuous by Proposition~\ref{prop:pointsetfun} (alternatively, note that $X/\DD_1$ is homeomorphic to $X$, since $\pi_1$ is approximable by homeomorphisms, and $X$ is a compact metric space). Moreover, $\pi_1(U)$ is open in $X/\mathcal{D}_1$, since $\pi_1^{-1}(\pi_1(U))=U$ is open in $X$, and contains all the non-singleton elements of $\mathcal{D}/\mathcal{D}_1$. We also note that $\pi_1$ is a uniformly continuous function by the Heine-Cantor theorem, and consequently $\mathcal{D}/\mathcal{D}_1$ is a null decomposition, since $\mathcal{D}$ is. 

We will next show that $\mathcal{D}/\mathcal{D}_1$ consists of starlike-equivalent subsets of $X/\mathcal{D}_1$. Since each $\Delta\in\mathcal{D}$ is recursively starlike-equivalent, we know from definition that $\Delta/\Delta_1$ is starlike-equivalent in $X/\Delta_1$. However, we need to show further that $\pi_1(\Delta)$ is starlike-equivalent in $X/\mathcal{D}_1$. 

Fix some $\Delta\in\mathcal{D}$. Observe that $\pi_1\colon X\to X/\mathcal{D}_1$ factors as a composition $\pi_1=\pi_1^b\circ \pi_1^a$ where $\pi^a_1\colon X\to X/\Delta_1$ and
$$\pi_1^b\colon X/\Delta_1\to (X/\Delta_1)/\{\pi_1^a(\Delta'_1)\mid \Delta'\in\mathcal{D}\setminus \{\Delta\}\}=X/\mathcal{D}_1$$
are quotient maps. By definition, there exists a closed neighbourhood $N(\pi_1^a(\Delta))\subset X/\Delta_1$ and a function $f\colon N(\pi_1^a(\Delta))\to D^n$ for some $n$ which is a homeomorphism onto its image and sends $\pi_1^a(\Delta)$ to a starlike set. Since $\pi_1^a$ is a uniformly continuous function by the Heine-Cantor theorem, the decomposition $\{\pi_1^a(\Delta'_1)\mid \Delta'\in\mathcal{D}\setminus\{\Delta\}\}$ is null in the compact metric space $X/\Delta_1$. By Lemma~\ref{lem:RSE-preserved}, we see that $\{\pi_1^a(\Delta'_1)\mid \Delta'\in\mathcal{D}\setminus\{\Delta\}\}$ consists of recursively starlike-equivalent sets of filtration length $N-1$ and thus by the inductive hypothesis, the function $\pi_1^b$ is approximable by homeomorphisms which agree with $\pi_1^b$ on the closed set $\pi_1^a(\Delta)$. Let $g$ be any such homeomorphism. We note $g(N(\pi_1^a(\Delta))$ is a closed neighbourhood of $g(\pi_1^a(\Delta))=\pi_1^b(\pi_1^a(\Delta))=\pi_1(\Delta)$. We also have the function $f\circ g^{-1}\colon g(N(\pi_1^a(\Delta))\to D^n$, which is a homeomorphism onto its image and maps $\pi_1(\Delta)$ to a starlike set by construction. This completes the argument that for each $\Delta\in\mathcal{D}$, we have that $\pi_1(\Delta)$ is starlike-equivalent in $X/\mathcal{D}_1$. 

We have now shown that the decomposition $\mathcal{D}/\mathcal{D}_1$ and the set $\pi_1(U)$ satisfy the hypotheses of Theorem~\ref{thm:SE-shrinking} in the space $X/\mathcal{D}_1$. As a result, the quotient map
$$\pi_0\colon X/\mathcal{D}_1\to (X/\mathcal{D}_1)/(\mathcal{D}/\mathcal{D}_1)=X/\mathcal{D}$$
is approximable by homeomorphisms restricting to the identity on $X/\mathcal{D}_1\setminus \pi_1(U)$. Then since the composition of maps which are approximable by homeomorphisms is itself approximable by homeomorphisms (Proposition~\ref{prop:composition-of-ABH}), we infer that $\pi\colon X\to X/\mathcal{D}$ is approximable by homeomorphisms restricting to the identity on $X\setminus U$ as desired. By induction, this completes the proof of the theorem. \end{proof}

\noindent The result needed for the proof of the disc embedding theorem in~\cite{FQ} is the following, which is an immediate corollary of Theorem~\ref{thm:main-theorem}.

\begin{theorem}\label{thm:alpha_ABH}
Let $\DD$ be a null decomposition of~$D^4$ consisting of recursively starlike-equivalent decomposition elements of filtration length one.
Suppose the decomposition elements are disjoint from the boundary of
$D^2 \times D^2$ \textup{(}although they may converge to it\textup{)} and further, that $S^1 \times D^2\subset \partial(D^2\times D^2)$ has a closed neighbourhood $C$
disjoint from the decomposition elements.
Then there is a homeomorphism
\[h\colon (D^2 \times D^2, S^1 \times D^2)\xrightarrow{\cong}((D^2 \times D^2) / \DD, S^1 \times D^2)\]
restricting to the identity on $C\cup \partial(D^2\times D^2)$.
\end{theorem}

\bibliographystyle{alpha}
\bibliography{shrinkingpaper_bib}
\end{document}